\documentclass[12pt,a4paper]{article}
\usepackage[english]{babel}
\usepackage[utf8]{inputenc}
\usepackage{amsmath,amssymb,amsthm,graphicx}
\usepackage{geometry}
\usepackage[all]{xy}
\usepackage{tikz}

\title{Stochastic modelling of Gaussian processes by improper linear functionals}

\author{Niels Lundtorp Olsen \\ Department of Applied Mathematics and Computer Science \\ 
	Technical University of Denmark}

\newcommand{\vektor}[1]{\begin{pmatrix} #1 \end{pmatrix}}
\newcommand{\eb}{\Leftrightarrow}
\newcommand{\pil}{\rightarrow}

\newcommand{\citep}{\cite}

\newcommand{\M}[1]{\mathbb{#1}}
\newcommand{\R}{\mathbb{R}}

\newcommand{\Hi}{\mathbb{H}}
\newcommand{\mb}{\mathcal{B}}
\newcommand{\iid}{i.i.d.\ }
\newcommand{\mvl}{mv.\ logL}

\newcommand{\for}{\text{ for }}

\newcommand{\brak}[1]{\langle #1 \rangle}
\newcommand{\bigbrak}[1]{\left\langle #1 \right\rangle}
\newcommand{\dotcomma}{, \dots ,}

\newcommand{\E}{\mathrm{E}}
\newcommand{\trans}{^\top \!} 
\newcommand{\I}{\mathbf{I}}

\newcommand{\de}{\: \mathrm{d}}
\newcommand{\inte}[4]{\int_{#1}^{#2} \! #3 \de \mathrm{#4} }

\DeclareMathOperator{\V}{V}
\DeclareMathOperator{\cov}{Cov}

\DeclareMathOperator*\uaf{\perp\mskip-11mu\perp}

\theoremstyle{definition}
\newtheorem{thm}{Theorem}
\theoremstyle{definition}

\theoremstyle{definition}
\newtheorem{defi}{Definition}
\theoremstyle{definition}
\newtheorem{proposition}[thm]{Proposition}
\theoremstyle{definition}
\newtheorem{lemma}[thm]{Lemma}
\theoremstyle{definition}

\theoremstyle{definition}
\newtheorem{eks}[thm]{Example}
\theoremstyle{definition}
\newtheorem{remark}[thm]{Remark}

\begin{document}
	
	\maketitle
	
	\begin{abstract}
Various approaches to stochastic processes exist, noting that key properties such as measurability and continuity are not trivially satisfied. 
We introduce a new theory for Gaussian processes using improper linear functionals. 
Using a collection of \iid standard normal variables, we define Gaussian white noise and discuss its properties. 
This is extended to general Gaussian processes on Hilbert space, where the variance is allowed to be any suitable operator. 
Our main focus is $L^2$ spaces, and we discuss criteria for Gaussian processes to be continuous in this setting. 
Finally, we outline a framework for statistical inference using the presented theory with focus on the special case of $L^2[0,1]$. We introduce the Fredholm determinant into the functional log-likelihood. 
We demonstrate that the naive functional log-likelihood is not consistent with the multivariate likelihood. A correction term is introduced, and we prove an asymptotical result.  
\end{abstract}

\medskip

\noindent%
{\it Keywords:} stochastic processes, white noise, $L^2$ spaces, functional data

\vfill

\newpage

\section{Introduction}	
The concept of treating a function as a random variable (ie. a \emph{stochastic process}) is central to many areas of statistics and probability theory, including diffusion theory, spatial statistics and functional data analysis. 
However, it is highly non-trivial to define a satisfactory mathematical framework for random variables on $\R^U$, for some continuous domain $U$, due to the uncountable index set.
Kolmogorov's extension theorem gives a general criterion for the existence of random variables on $\R^U$, but does not answer some of basic questions that naturally arise when these random variables are viewed as functions $U \pil \R$: when do have measurability, continuity etc?, which require additional assumptions. It is for instance well-known that a white noise process on $\R^U$ cannot be measurable  as a function  $U \pil \R$, which has motivated alternative definitions based on tempered distributions \citep{hidabog}. 
A good treatment of the theory of random variables on continuous domains can be found in \cite{taylorAdler}. 

The majority of work on stochastic processes has been done within the framework of Gaussian processes (ie. when all marginal distributions are Gaussian), or when Gaussian variables make good approximations (e.g. high-frequency diffusion processes).
That is not necessarily because we believe the Gaussianity to be true for actual applications, but Gaussian random variables are well understood and may often be what is feasible.

Gaussian processes are commonly defined as elements of 
a suitable Hilbert space. In functional data analysis (FDA), the default choice is the space of square-integrable functions, $L^2[U]$, or some suitable subset thereof, e.g.\ the set of bounded continuous functions $C(U)$. 
This for instance plays a significant role in popular FDA textbooks \citep{ramsay, horvathkokozka}.  

A related and important class of stochastic processes are \emph{diffusion processes}, with many diverse applications \cite{oksendal}.
A diffusion process can be interpreted as a process which "accumulates" white noise, adding randomness to ordinary differential equations.
This accumulation is not straightforwardly defined, and there are two competing approaches to stochastic integration, the \emph{Ito integral} (more common) and the \emph{Stratonovich integral}. 

A general result states that the covariance function of a continuous Gaussian process defined on an interval $[a,b]$, is a positive semi-definite and compact operator in $L^2[a,b]$. 
However, operators 'lack' two interesting properties that finite-dimensional operators possess: (1) non-invertability, (2) any suitable generalization of the determinant would give zero (as a consequence of (1)).
A generalisation of determinants to function spaces was introduced by Fredholm \citep{fredholm1903} as early as 1903. Whereas we do find references to the Fredholm determinant in the statistical literature, we have not seen it used as a 'penalisation' term like the determinant of a multivariate Gaussian likelihood. 

Many of the useful results about continuous stochastic processes can be found in \cite{berlinetAgnan}, which also discuss the association with \emph{reproducing kernel Hilbert space} (RHKS) a powerful concept with diverse applications.

Finally, we remark that data are more or less noisy, contrary to stochastic process models which generally assume smoothness or at least continuity. Although researchers acknowledge this fact, we believe this to be a relevant shortcoming in many cases. This is particularly relevant for asymptotical results on decreasing sampling distance, where noise is rarely considered.

\subsection{Contribution of this article}
This paper contributes to the foundational framework of random functions  with a new definition of Gaussian processes that has a mathematically quite simple definition.
The major innovation is to allow a Gaussian process to be an \emph{improper} functional. This 'relaxation' opens up new possibilities and interpretations of data, for 
instance the use of non-compact operators for statistical modelling, and we argue that this fits well with the actual, noisy nature of data.

The remainder of the paper is organised as follows: 
In Section \ref{gw-basic} we introduce the Gaussian white noise, which is extended to 'general' Gaussian processes in Section \ref{gaus-ops}. Section \ref{l2-sec} concerns Gaussian processes for $L^2$-spaces and discusses properties such as continuity.  
Finally, Sections \ref{inference-sec} and \ref{l201} outlines statistical inference. Section \ref{l201} considers the specific case of $L^2[0,1]$. Here we introduce the Fredholm determinant and prove asymptotical equivalence of the functional and multivariate Gaussian likelihoods. A short discussion on the findings of this article is provided in Section \ref{sect-disc}.

\begin{remark}[Notational conventions]
	All vector spaces in this article are assumed to be \emph{real vector spaces} as we work with 'real' data. Furthermore, all vector spaces shall be of countably infinite dimension unless otherwise stated, and $\brak{ \cdot, \cdot}$ is used to denote the inner product. In particular, all Hilbert spaces shall be separable.
	
	We will use the common abbreviations a.s.\ for '{almost surely}' and \iid for '{independent and identically distributed}'.
\end{remark}

\section{Gaussian white noise} \label{gw-basic}
In this section we introduce Gaussian white noise and show its basic properties. 
We will introduce the theory on general separable Hilbert spaces, but our main focus is  $L^2$-spaces, where we in Section \ref{gaus-ops} will use operators.

Notably,  instead of associating random variables elements in a Hilbert space, we shall associate random variables with a functional. We show that the functionals are improper almost surely.
This explicitly violates the basic assumptions of functional analysis and allow us to go beyond the limitations imposed by Riesz' representation theorem \citep{Riesz}. This theorem states that functions and functionals are dual properties 
-- there is a one-to-one mapping  $\phi: \Hi \pil \Hi^*$ given by: 
\[ \phi(g) =  \left[f \mapsto \brak{g,f}, f \in \Hi \right]
\]	
where 
$\Hi^*$ is the associated space of bounded linear functionals. This arguably limits the usefulness of functionals in statistical modelling, unless additional properties for $\Hi$ are considered.

\subsection{Gaussian white noise}
We begin by defining the concept of Gaussian noise on an infinite-dimensional Hilbert space. 

\begin{defi} \label{gauss-defi}
	Let $\Hi$ be a separable Hilbert space with an orthonormal basis $(e_i)_{i=1}^\infty$. Let $\{X_i\}_{i = 1}^\infty$ be a sequence of \iid $N(0,1)$-distributed random variables defined on some background probability space $(\Omega, \M{F}, P)$. 
	
	The \emph{Gaussian white noise} $W$ is defined as the random variable $W: \Hi \times \Omega \pil \R \cup \{\infty \}$
	\begin{equation}
	W(f)(\omega) = \sum_{i = 1}^\infty  \brak{ f, e_i} X_i(\omega), \quad f \in  \Hi, \omega \in \Omega
	\end{equation}
	with the convention that $W(f)(\omega) = \infty$ when the series does not converge. 
\end{defi}
Thus, the application of $W$ to a function $f\in \Hi$ outputs a Gaussian random variable -- and $W$ applied to a different function $g \in \Hi$ would output a different random variable. 
It holds that $(W(f) , W(g))$ is jointly Gaussian, and the covariance is easily inferred:

\begin{proposition} \label{prop-gauss}
	Let $W$ be a Gaussian white noise. Then:
	
	(a) $(W(f_1), \dots , W(f_n))$ follows a multivariate Gaussian distributed with mean zero and 
	\begin{equation}
	\cov(W(f_i), W(f_j)) = \brak{f_i, f_j}
	\end{equation}
	where $f_1, \dots, f_n \in \Hi$.
	
	(b) The distribution of $W$ is independent of the choice of basis for $\Hi$.
\end{proposition}

It follows from (a) that $W(f)(\omega) \in \R$ almost surely (a.s.), thus we should not really worry about when the series is non-convergent.

\begin{proof}

	As $\sum_{i = 1}^\infty \V[\brak{ f, e_i} X_i ] = \sum_{i = 1}^\infty \brak{ f, e_i}^2  = ||f||^2 < \infty$, it follows that $W(f)$ is a well-defined real number almost surely for any $f \in \Hi$. 
	Now observe that $W(f_1), \dots , W(f_n)$ are all convergent series defined by linear transformations using the same underlying \iid normal variables, thus $(W(f_1), \dots W(f_n))$ is jointly Gaussian. 
	
	It is easily seen that $\E[W(f)] = 0$, and from the $X_i$s being i.i.d. and standard Hilbert space theory we get
	\begin{equation}
	\cov(W(f), W(g)) = \sum_{i = 1}^\infty \brak{e_i, f} \brak{e_i, g} = \brak{f, g}, \quad f,g \in \Hi
	\end{equation}
	proving the first part. 
	As none of the variance or mean expressions involve $(e_i)_{i=1}^\infty$, we see that the distribution of $W$ is independent of the chosen basis for $\Hi$.
\end{proof}

\begin{remark} 
	Gaussianity is crucial for the second statement of Proposition \ref{prop-gauss}. It is well-known that a vector $(X_1 \dotcomma X_N)$ of \iid random variables is invariant to rotations (here: change of basis) iff it is jointly Gaussian. 
	The construction in Definition \ref{gauss-defi} would work for other distributions than the Gaussian, but the invariance property is then lost.  
\end{remark}	

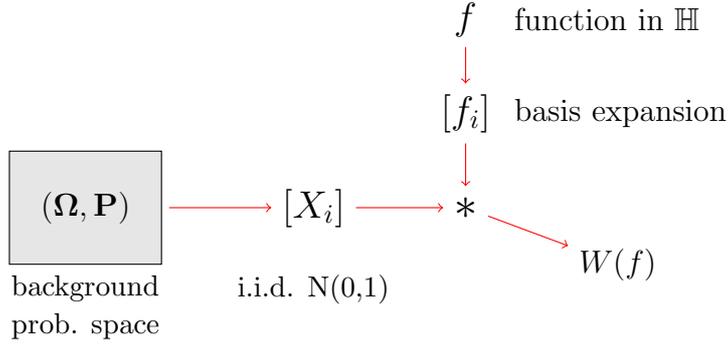
\begin{figure}
	\[
	\begin{tikzpicture}
	\draw [fill=gray!20, draw=black] (0,0) rectangle (2,1.5)
	(1,0.75) node {($\mathbf{\Omega, P}$)}
	(1,0) node[below, align = center] 
	{\small background \\ \small prob. space}
	(4,0.75) node (x) {\large $\left[X_i\right]$}
	(4,0) node[below] {\small i.i.d. N(0,1)}
	(6, 0.75) node (a) {\Large $\ast$}
	(6, 3.25) node (f) {\large $f$}
	(6, 2) node (i) {\large $\left[f_i\right]$}
	(6.5, 3.25) node[right] {function in $\Hi$}
	(6.5, 2) node[right] {basis expansion}
	(8,0) node (w) {$W(f)$}
	;
	\draw[->, red ] (2.1,0.75) -- (x);
	\draw[->, red ] (x) -- (a);
	\draw[->, red ] (f) -- (i);
	\draw[->, red ] (i) -- (a);
	\draw[->, red ] (a) -- (w);
	
	\end{tikzpicture}
	\]
	
	\caption{Application of white Gaussian noise on $f \in \Hi$}
\end{figure}

The second important proposition is linearity:
\begin{proposition}[Linearity]
	$W$ is linear. That is, for $f, g \in \Hi$, and $\lambda \in \R$, then 
	\begin{equation}
	W(\lambda f + g) = \lambda W(f) + W(g)	\quad \text{almost surely}
	\end{equation}
\end{proposition}
\begin{proof}
	Follows from linearity of the inner product, ie. $\brak{\lambda f + g, e_i} = \lambda \brak{f, e_i} +  \brak{g, e_i}$.
\end{proof}

So we should ask ourselves: what kind of objects are instances of Gaussian white noise $W$? They seem to resemble linear functionals; the domain is $\Hi$;  $W(f)$ is a real number almost surely; and we have linearity.

However instances of $W$ are {not} proper linear functionals, as shown in the following proposition: 

\begin{proposition} \label{gw-non-lin}
	It almost surely holds that $W$ is an improper unbounded linear functional, in the sense that for almost-all $\omega \in \Omega$, 
	
	(a) there exists a sequence of unit vectors $f_1, f_2, \dots \in \Hi$ such that $W(f_k)$ is finite for all $k$, but $\lim_{k \pil \infty} W(f_k)(\omega) = \infty$.
	
	(b) there exists $g \in \Hi$ such that $W(g)(\omega) = \infty$. 
\end{proposition}

\begin{proof}
	First note that $ (X_i)$ is an {infinite} sequence of \iid standard normal variables.
	With probability one, for every $x > 0$ we can find an infinite number of $i$'s s.t. $X_i > x$. Thus, there exists a sequence $(i_k)_{k=1}^\infty$ for which $X_{i_k} > k$ for all $k$. Now define $f_k = e_{i_k}$ for all $k$. Then $W(f_k) = X_{i_k} > k$ for all $k$, proving the first part. 
	
	Assume we are given an $\omega \in \Omega$ s.t. the first part holds. Now define 
	\begin{equation}
	g = \sum_{k=1}^\infty k^{-2} e_{i_k} \in \Hi
	\end{equation}
	Then $W(g)(\omega) = \sum_{k=1}^\infty k^{-2} X_{i_k}(\omega) > \sum_{k=1}^\infty k^{-1} = \infty$, proving the second part.  
\end{proof}
The fact that we are allowed to conclude (b) despite that $W(g) \in \R$  a.s. for any $g$, is related to the fact that the number of subsequences of $e_{i_k}$ is uncountable. 
This can also be seen as an instance of a more general aspect of measure theory, namely that:
\begin{align*}
[(y \in A_y) \quad \text{almost surely}] \quad \text{for all } y \in \mathcal{Y} \\
[(y \in A_y) \quad \text{for all }y  \in \mathcal{Y}] \quad \text{almost surely}
\end{align*}
are two very different statements when $\mathcal{Y}$ is an uncountable set.

\subsection{Integration of white noise} \label{white-set-sec}
The concept of Gaussian noise is arguably most interesting when  $\Hi$ is a function space of square-integrable functions, ie. $\Hi = L^2[U]$ for some $U \subseteq \R^m$.
We will use $\mu$ to denote the Lebesgue measure and $\mathbb{B}_U$ to denote the set of Borel sets on $U$.

We shall identify sets with their indicator functions in the sense that $W(A) := W(1_A)$ for $A \in \mathbb{B}_U$. 

\begin{proposition}
	For $A,B \in \mathbb{B}_U$ we have
	\begin{align}
	W(A) \sim N(0, \mu(A)) \\
	\mu(A \cap B) = 0 \eb W(A \cup B) = W(A) + W(B) a.s. \eb  W(A) \uaf W(B)
	\end{align}
	and the distribution of $W$ is independent of the choice of basis for $L^2(U)$.
\end{proposition}
\begin{proof}
	$(A \cap B$ is null set$) \Rightarrow 1_{A \cup B} = 1_A + 1_B$ a.s. $\Rightarrow W(A \cup B) = W(A) + W(B)$. Regarding the opposite implication, it is a small exercise in writing down covariances and deducing that $\mu(A \cap B) = 0$. Finally, as everything is jointly Gaussian, $ W(A) \uaf W(B) \eb \cov(W(A), W(B)) = 0 \eb \brak{1_A, 1_B} = 0 \eb \mu(A \cap B) = 0$. 
\end{proof}
Thus $A$ and $B$ are disjoint iff $W(A)$ and $W(B)$ are independent -- this is intuitive from the 'white noise' point of view and agrees with the definition of \cite[Chapter 1.4.3]{taylorAdler}.

If $A$ is an interval on the real line, this defintion aligns with stochastic integration of pure white noise, ie.: 
$$
\int_A \de U_t = W(A)
$$
where $U$ is a Wiener process. 

One may think of a realization of $W$ as an element of $L^2[U]$ consisting of "pure noise". However, strictly speaking this does not make sense: by Proposition \ref{gw-non-lin} and the non-applicability of Riesz' theorem, we cannot identify $W$ with a function in $L^2[U]$. 

When operators have been introduced, we will have the Brownian motion as an example in Section \ref{inference-sec}.

\section{Gaussian processes defined using operators on Hilbert space} \label{gaus-ops}

The idea of 'white  noise' is a somewhat abstract mathematical concept when viewed on its own.
Here we will introduce operators which allow us to define a much broader class of random variables, where the identity operator, $\I$, corresponds to the Gaussian white noise of the previous section. 

Many authors have used operators in the context of Gaussian processes, where the covariance function is identified with an integral operator. By the Karhunen–Loève theorem, such operators have to be compact for the resulting function to be continuous. A notable paper that uses more general operators in relation to Gaussian processes is \cite{markussen2013}. 

\paragraph{Notation} 
Let $B(\Hi, \M{K})$ denote the class of bounded linear operators $\Hi \pil \M{K}$ for separable Hilbert spaces $\Hi, \M{K}$. 
$O^* \in B(\M{K}, \Hi)$ will denote the adjoint of $O$, and $B(\Hi)$ will be shorthand for  $B(\Hi, \Hi)$.

\begin{defi}[Gaussian process defined by the operator $O$]
	Assume $\mathbb{H}, \M{K}$ to be separable Hilbert spaces, where $(e_i)_{i = 1}^\infty$ is an orthonormal basis for $\Hi$.
	
	Let $O \in B(\Hi, \M{K})$. 	The random variable $\tilde{O}: \M{K} \times \Omega \pil \R$ is defined by:
	\begin{equation}
	\tilde{O}(f)(\omega) := \sum_{i = 1}^\infty X_i(\omega) \brak{ f, O e_i}, \quad f \in  \mathbb{K}, \omega \in \Omega
	\end{equation}
	which for any $f$ exists a.s. We define variance of $\tilde{O}$ to be the operator $OO^* \in B(\M{K})$.  
	
	The variance of $\tilde{O}$ is defined to be $OO^* \in B(\M{K})$. 
\end{defi}

\begin{proposition}
	The distribution of $\tilde{O}$ is independent of the choice of basis for $\Hi$.
	$(\tilde{O}(f_1), \dots, \tilde{O}(f_n))$ is multivariate normally distributed with mean zero and  $\cov(\tilde{O}(f_i), \tilde{O}(f_j)) = \brak{O^*f, O^*g}$.
\end{proposition}

\begin{proof}
	The proof is similar to the proof of Proposition \ref{prop-gauss}. 
	The only thing of interest to show is
	\begin{equation}
	\cov(\tilde{O}(f), \tilde{O}(g)) = \sum_{i = 1}^\infty  \brak{f,  O e_i} \brak{g,  O e_i}  = 
	\sum_{i = 1}^\infty  \brak{O^* f,  e_i} \brak{O^* g,  e_i} = \brak{O^*f, O^* g}
	\end{equation}
\end{proof}
Linearity  of operators is  preserved by this construction: if $R = \lambda O + K$ for $\lambda \in \R$ and $O,K \in B(\Hi, \M{K})$, then:
\[
\tilde{R}(f) = \lambda \cdot \tilde{O}(f) + \tilde{K}(f).
\]
If given an operator $O$  in $B(\Hi, \M{K})$, we shall view $OO^*$ as the variance of $\tilde{O}$. The latter is a positive semidefinite operator which uniquely characterizes the distribution of $\tilde{O}$:  $\tilde{O} \stackrel{D}{=} \tilde{R} \eb OO^* = RR^*$. The standard formulae for covariances follow from linearity of operators.
\begin{figure}
	\[
	\begin{tikzpicture}
	\draw [fill=gray!20, draw=black] (0,0) rectangle (2,1.5)
	(1,0.75) node {($\mathbf{\Omega, P}$)}
	(1,0) node[below, align = center] 
	{\small background \\ \small prob. space}
	(4,0.75) node (x) {\large $\left[X_i\right]$}
	(4,0) node[below] {\small i.i.d. N(0,1)}
	(6, 0.75) node (a) {\Large $\ast$}
	(6, 3.25) node (f) {\large $f$}
	(6, 2) node (i) {\large $\left[f_i\right]$}
	(4.75, 3.25) node (o) {${O}$}
	(4.6, 3.1) node[left] {Operator in $\mb(\Hi)$}
	(6.5, 3.25) node[right] {function in $\Hi$}
	(6.5, 2) node[right] {basis expansion of $O^* f$}
	(8,0) node (w) {$\tilde{O}(f)$}
	;
	\draw[->, red ] (2.1,0.75) -- (x);
	\draw[->, red ] (x) -- (a);
	\draw[->, red ] (f) -- (i);
	\draw[->, red ] (i) -- (a);
	\draw[->, red ] (a) -- (w);
	\draw[->, red ] (o) -- (i);
	
	\end{tikzpicture}
	\]
	
	\caption{Application of a Gaussian process $\tilde{O}$ on $f \in \Hi$}
\end{figure}
\medskip

Which kind of objects are instances of $\tilde{O}$?
Like white Gaussian noise, $\tilde{O}$ is an (improper) linear functional, and it is trivially seen that if $O = \I$, the identity operator, then $\tilde{O}$ is Gaussian white noise. 
However, there is a plethora of interesting operators, so we have many  opportunities for interesting Gaussian processes. 
The following proposition contrasts with Proposition \ref{gw-non-lin} and gives a criterion for instances of $\tilde{O}$ to be proper linear functionals. By Riesz Representation theorem, these can further be identified with elements in $\Hi$. 

\begin{thm} \label{hilbert-schmidt}
	Let $O \in B(\Hi)$ be a Hilbert-Schmidt operator. Then $\tilde{O}$ is a bounded linear functional for almost all $\omega \in \Omega$.
\end{thm}
\begin{proof}
	Let $(O_{ij})_{i,j = 1}^\infty$ be the matrix representation for $O$ corresponding to $ (e_i)_{i=1}^\infty $. By the Hilbert-Schmidt assumption, $\sum_{i,j} O_{i,j}^2 < \infty$.
	Since  $X_i$ is a sequence of \iid N(0,1) variables, $\sum_{i = 1}^\infty \sum_{j = 1}^\infty O_{ij}^2  X_i^2$ is finite a.s.
	Therefore, the operator defined by $e_i \mapsto \sum_{i = 1}^\infty X_i O e_i$ is Hilbert-Schmidt almost surely, and the associated functional
	\begin{equation}
	\tilde{O}(f) := \sum_{i = 1}^\infty X_i(\omega) \brak{f,  O e_i} = \brak{f,  \sum_{i = 1}^\infty X_i(\omega) O e_i}
	\end{equation}	
	is  bounded for almost all $\omega \in \Omega$.
\end{proof}

\begin{lemma} 
	Note that compactness of $O$ is not a sufficient condition for $\tilde{O}$ to be a linear functional. 
	
	This can be seen choosing $O \in B(\Hi)$ s.t. $Oe_i = \phi_i e_i$, where $\phi_i \pil 0$ for $i \pil \infty$, and $\phi_i \geq  1 / \sqrt{\log \log i}$ for $i \geq 3$. 
	Then $O$ is a compact operator, but $\tilde{O}$ is not bounded:
	
	For $u > 1$, it holds that $P(X_i > u) > e^{-u^2 / 2} / u$. Therefore, for $x > 1$:
	\begin{equation}
	\sum_i P(X_i > \frac{x}{\phi_i}) > \sum_i \frac{\phi_i e^{-x^2 /(2\phi_i^2))}}{x} \geq \sum_i \frac{ e^{-\frac{1}{2}x^2 \log \log i}}{x \sqrt{\log \log i}} = \sum_i \frac{ 1 }{x \sqrt{\log \log i} (\log i)^{x^2 /2}} = \infty 
	\end{equation}
	By the Borel-Cantelli lemma it happens a.s. that $X_i > \frac{x}{\phi_i}$ infinitely often. Thus a.s, there exists a sequence $(i_k)_{k = 1}^\infty$ s.t. $X_{i_k} > \frac{k}{\phi_{i_k}}$ for all $k$. Now define $f_k = e_{i_k}$ for all $k$. Then $||f_k|| = 1$ and $\tilde{O}f_k = X_{i_k} \phi_{i_k} > k$. This proves that $\tilde{O}$ is unbounded a.s.
\end{lemma}
The convergence of $\phi_i$ in the above lemma is very slow, and the result  depends intrinsically on the tail probabilities for the normal distribution.

\subsection{Inverse and precision operators}

If $O \in B(\Hi)$ is invertible, then we can speak of a \emph{precision operator} and the inverse $O^{-1}$ converts $\tilde{O}$ back to white Gaussian  noise. The precision operator is seen to be $O^{*,-1} O^{-1}$.

The precision operator can be used as a norm on $\Hi$, giving a distance between data in $\Hi$ and a quantification of outliers.
In Section \ref{l201} we will use the inverse operator  to define a functional log-likelihood for statistical modelling.

\section{Operators on $L^2$-spaces} \label{l2-sec}
In the following we will consider the special case $\Hi = L^2[U]$, where $U$ is an open subset of $\R^m$. 
By Theorem \ref{hilbert-schmidt}, Hilbert-Schmidt operators are true linear functionals, and thus we can associate a realisation of $\tilde{O}$ with a function $f: U \pil \R$. 
We will however define $\tilde{O}(x)$ as the downward limit of  $\tilde{O}$ applied to balls around $x$ for $x \in U$. As in Section \ref{white-set-sec}, we shall identify sets with their indicator functions, ie. $\tilde{O}(A) := \tilde{O} 1_A$ for $A \in \mathbb{B}_U$.

\begin{defi}[Local continuity]
	An operator $O \in L^2(U)$ and its resulting random variable $\tilde{O}$ is  \emph{locally continuous} in a point $x \in U$ if it holds that $\tilde{O}(B(x, \epsilon)) / \mu(B(x, \epsilon))$ is $L^2$-convergent for $\epsilon \pil 0$, where $B(x, \epsilon)$ is the ball around $x$ with radius $\epsilon$. 
\end{defi}
We define this limit to be the value of $\tilde{O}(x)$. 
It is easily seen that if $\tilde{O}$ is locally continuous in $x_1 \dotcomma x_n$, then $\tilde{O}(x_1) \dotcomma \tilde{O}(x_n)$ are jointly Gaussian, and
$$
\cov(\tilde{O}(x_1), \tilde{O}(x_2)) = \lim_{\epsilon \pil 0} \bigbrak{\frac{O^* 1_{B(x_1, \epsilon)}}{ \mu(B(x_1, \epsilon))} ,  \frac{O^* 1_{B(x_2, \epsilon)}}{ \mu(B(x_2, \epsilon))}}
$$
Furthermore, we define $\tilde{O}$ to be continuous if $\tilde{O}$ is locally continuous for every $x \in U$. If  $\tilde{O}$ is continuous, we can thus associate $\tilde{O}$ with a zero-mean Gaussian process in the 'classical' sense (ie. a process associated with a covariance function $U \times U \pil \R_+$). 
As we discuss in Remark \ref{kont-remark}, this covariance function is simply the kernel of $OO^*$, under regularity conditions.

For a large class of integral operators, it holds that $\tilde{O}$ is continuous.
Contrary, the Gaussian white noise  process $W$ is nowhere locally continuous. This is also true for the sum of a white noise Gaussian process and a continuous Gaussian process. 

\subsection{Integral operators and continuity} \label{int-sec}

\begin{lemma} \label{intop}
	Let $K$ be an integral operator with kernel $K(t,s)$, ie.\ $Kf(t) = \int_U K(t,s) f(s) \de s$. 
	If it for $t_0 \in U$ holds that 
	\begin{equation}
	t \mapsto \mu(B(t_0, \epsilon))^{-1} \int_{B(t_0, \epsilon)} K(s,t) \de s \label{L2-graense}
	\end{equation}
	converges in $L^2$ for $\epsilon \pil 0$, then the associated Gaussian process $\tilde{K}$ is locally continuous in $t_0$.
\end{lemma}
\begin{proof}
	Let $(\epsilon_n)_{n=1}^\infty$ be a positive sequence converging to zero.
	By definition, 	$\tilde{K}$ is locally continuous in $t_0$ if $\frac{K^* 1_{B(t_0,\epsilon_n)}}{ \mu(B(t_0, \epsilon_n))}$ is convergent. 
	Since $ \int_A K(s,t) \de s = (K^*1_A)(t)$ for any $A \in \mathbb{B}_U$, this is true by assumption. 
\end{proof}

\begin{proposition} \label{kont-korr}
	Assume $K$ is uniformly continuous on $U$, and assume $\mu(U)$ is finite. Then $\tilde{K}$ is well-behaved for all $t_0 \in U$, and $\cov(\tilde{K}(u_1) , \tilde{K} (u_2)) = \int_U K(u_1, t) K(u_2, t) \de t$
\end{proposition}
\begin{proof}
	We show that the expression in \eqref{L2-graense} converge to $K(t, t_0)$.
	
	By uniform continuity it holds that for any $\rho > 0$ there is a $\epsilon_0$ s.t. $|K(t,s) - K(t, t_0)| < \rho$ for $s \in B(t_0, \epsilon)$, $t \in U$ and $\epsilon < \epsilon_0$. Then:
	\begin{multline*}
	\left| \mu(B(t_0, \epsilon))^{-1} \int_{B(t_0, \epsilon)} K(s,t) \de s - K(t_0, t) \right| = \\
	\left| \mu(B(t_0, \epsilon))^{-1} \int_{B(t_0, \epsilon)} K(s,t) -  K(t_0, t)\de s \right| < \rho
	\end{multline*}
	for all $t$. 
	Hence, $$\left|\left|t \mapsto \mu(B(t_0, \epsilon))^{-1} \int_{B(t_0, \epsilon)} K(t,s) \de s - K(t, t_0) \right|\right| < \rho \mu(U),$$ which proves the first part. 
	
	We have that, 
	\begin{equation}
	\cov(\tilde{K}(u_1) , \tilde{K} (u_2))  = \bigbrak{\lim_{\epsilon \pil 0} \frac{{K}^*{1_{B(u_1, \epsilon)}} }{\mu(B(u_1, \epsilon))}, 
		\lim_{\epsilon \pil 0} \frac{{K}^*1_{B(u_1, \epsilon)} }{\mu(B(u_2, \epsilon))}}
	\end{equation}
	The left part converges uniformly to the function $t \mapsto K(u_1, t)$, and the right part converges uniformly to $t \mapsto K(u_2, t)$. This gives the result. 
\end{proof}

For, functions on the unit interval, $L^2[0,1]$, an 
important class of operators are \emph{triangular operators}; that is, integral operators $K$ where the kernel satisfies $K(t,s) = 0 $ for $ s > t$. Since these in general not are continuous at the diagonal, Proposition \ref{kont-korr} does not apply in this case, although we get a similar result (proof omitted):

\begin{proposition} \label{triang-korr}
	Let $K$ be a triangular operator on $L^2[0,1]$ s.t. the kernel is (uniformly) continuous on the closed lower triangle $L = \{(x,y) \in [0,1]^2 | x \geq y \}$.
	Then $\tilde{K}$ is locally continuous for all $t_0 \in (0,1)$, and 
	$\cov(\tilde{K}(t_1) , \tilde{K} (t_2)) = \inte{0}{t_1 \wedge t_2}{K(t_1, t) K(t_2, t)}{t}$.
\end{proposition}

One of the main motivations for working with triangular operators is their association to  stochastic integrals. If we define the diffusion process $Y_t$ as
$$
Y_t = \int_0^t K(t,s) \de W_s,
$$
then $Y_t$ would yield the same covariance expression as in Corollary \ref{triang-korr}.

\begin{remark} \label{kont-remark}
	From Proposition \ref{kont-korr} it follows that the covariance function of $\tilde{K}$ as a function on $U \times U$ coincides with the kernel of $KK^*$, which we in Section \ref{gaus-ops} defined as the variance of $\tilde{K}$. 
	Similarly holds for the triangular operators of Proposition \ref{triang-korr}.
	
	This results  links 'classical' Gaussian process  to our definition of a Gaussian process: the variance of
	$\tilde{K}$ is 'identical' when $\tilde{K}$ is viewed as a function $U \pil \R$.  
\end{remark}

\subsection{Examples of Gaussian processes on $L^2$-spaces}

\begin{eks}[Brownian motion] \label{bm-eks}
	The forward integral operator on $[0,1]$, which is a triangular operator, may be used to construct a Wiener process. Define $K$ as the integral operator with kernel: 
	\begin{equation}
	K(t,s) = \lambda 1_{s < t}
	\end{equation}	
	Then $\tilde{O}$ is the Wiener process with variance parameter $\lambda^2$. This can be seen by applying Corollary \ref{triang-korr}.
\end{eks}

This example illustrates how  the Brownian motion/Wiener process can be interpreted as 'cumulated Gaussian noise'. This is an intuitive way of viewing the Brownian motion that also applies to many people's understanding of the Brownian motion.

\begin{eks} \label{bb-noise}
	Let $O = K + \mathbf{I}  \in B(L^2[0,1])$, where $\mathbf{I}$ is the identity operator, and $K$ is an integral operator with kernel $K(s,t) = \min(s,t) - st$.  $K$ is the covariance for a Brownian bridge, so $O$ is the covariance operator for a Brownian bridge with noise. 
	Since $O$ is positive definite, it has a unique positive definite square root $R$ s.t. $O = R^2$, and $O$ would therefore be the covariance of $\tilde{R}$. It is known that the Brownian bridge $K$ has eigenfunctions $\{\sin k \pi t\}_{k = 1}^\infty$ and eigenvalues $\{k^{-2} \pi^{-2}\}_{k = 1}^\infty$. We can now construct $R$; it has the same eigenfunctions as $K$ and eigenvalues $\{\sqrt{1 + k^{-2} \pi^{-2}}\}_{k = 1}^\infty$.
	
	Note that $R$ is not compact and thus $\tilde{R}$ is an improper functional a.s.  
\end{eks}	
This example can generalized to any positive semidefinite operator, showing how to construct a Gaussian process from its covariance.  

\begin{eks}[Bivariate Gaussian processes] \label{bivar-gp}
	
	We can construct bivariate Gaussian processes on $L^2[U]$ by identifying $L^2[U; \R^2]$ with $L^2[U] \oplus L^2[U] \cong L^2[U_1 \cup U_2]$ where $U_1$ and $U_2$ are  disjoint copies of $U$. 
	
	Let $\tilde{O}$ be a Gaussian process on 	$L^2[U] \oplus L^2[U]$ where $O = \vektor{O_{11} & O_{12} \\O_{21} & O_{22}}$, $O_{11} , O_{12} , O_{21} , O_{22} \in \mathcal{B}(L^2[U])$ is the matrix decomposition. 
	The restriction of $\tilde{O}$ to functions in $ L^2[U_1 \cup U_2]$  with support in $U_1$ can be viewed as a Gaussian process $\tilde{O}_1: L^2[U] \pil \R$ with variance $O_{11}O_{11}^* + O_{12} O_{21}^*$. We can similarly define $\tilde{O}_2$, making the joint variable $(\tilde{O}_1, \tilde{O}_2)$ a bivariate Gaussian process. 
	It holds that:
	\[
	\cov(\tilde{O}_1(f), \tilde{O}_2(g)) = \brak{f, (O_{11} O_{12}^* + O_{21} O_{22}^*) g}, \quad f,g \in L^2[U]
	\]
	Here $O_{11} O_{12}^* + O_{21} O_{22}^* \in \mathcal{B}(L^2[U])$ is interpreted as the covariance between $\tilde{O}_1$ and $\tilde{O}_2$.
	
\end{eks}
The construction in Example \ref{bivar-gp} can  be extended to higher dimensions, allowing for general multivariate Gaussian processes. 

\begin{eks}[Ornstein-Uhlenbeck process]
	The triangular operators of Section \ref{int-sec} are defined on the unit interval, or more precisely $L^2[0,1]$. However, with a little care we can define triangular operators on $L^2[\R]$ and use it to construct a stationary Markov Gaussian process on $L^2[\R]$. 
	
	Let $O \in \mathcal{B}(L^2[\R])$ be the integral operator with kernel 
	$$ O(t,s) = \alpha e^{\lambda (t-s)} 1_{t > s}$$
	where $\alpha, \lambda > 0$. Then $\tilde{O}$ is locally continuous and 
	$$\cov(\tilde{O}(t_1), \tilde{O}(t_2)) =  \tfrac{\alpha^2}{2\lambda} e^{-\lambda|t_1 - t_2|}$$
	which is the covariance function for the \emph{Ornstein-Uhlenbeck process}. \medskip
	
	Note that although 	 $\tilde{O}$ is locally continuous,
	it is a.s.\ not a proper linear functional and the {Ornstein-Uhlenbeck process} is not an element in $L^2[\R]$. In fact, $O$ is not even compact. 
\end{eks}

\section{Observational data and inference} \label{inference-sec}
Despite the continuity assumptions often assumed for stochastic process models, most observed (discrete) data are in fact noisy and must go through some kind of pre-processing to output a continuous function. In this section we discuss inference and the association of (discretely) observed data to the presented framework, emphasizing that noise is not considered a nuisance when the covariance operator is non-compact. 

We will sketch a general framework for statistical inference for a general Hilbert space $\Hi$, before continuing to the special case of $L^2[0,1]$ in Section \ref{l201}. 
\medskip

Data from time series and FDA usually come as one or more families of discrete observations. For simplicity, we will here assume one such family, $\{(u_k, y_k)| u_k \in U, y_k \in \R\}_{k = 1}^m$ which has been centered, ie. $\E[y_k] = 0$.

If we let $\tilde{O}$ be a Gaussian process on the Hilbert space $L^2[U]$, then $\tilde{O} (\omega)$ is an (improper) functional and so  takes {functions} as inputs, not numbers $x \in U$. 
If we were to define $y_k = \tilde{O}(u_k)$, this would require  $\tilde{O}$ to be locally continuous. 
We shall therefore make a functional embedding(s) of $y_k$; ie. $y_k \mapsto f_k y_k \in L^2[U]$. Here we for instance could choose $f_k$ to be indicator functions, so the embedding would make a piecewise constant approximation of $y$.

\subsection{Inference}
For our inferential framework, we shall require the following elements: (1) a statistical model, which is a family of positive (semi-)definite operators $\mathcal{K} \subseteq B(\Hi)$,
(2) a linear embedding of data into a function $f \in L^2[U]$,
$$
f = \sum_{k = 1}^{m} f_k y_k,
$$
and (3) a log-likelihood expression on the form 
\begin{equation} \label{loglik-gen}
Q(f, K) - \phi(K), \quad K \in \mathcal{K}, f \in L
\end{equation}
where $Q(\cdot, K)$ is a quadratic form in $f$. 
\medskip

The 'faithful' choice for \eqref{loglik-gen} would be to use the fact that $\cov(\tilde{O}(f_k), \tilde{O}(f_l)) = \brak{O^* f_k, O^* f_l}$. 
This would allow us construct a variance matrix for $\mathbf{y} = (f_1y_1, \dots , f_my_m)$, so that the  multivariate log-likelihood fits into  \eqref{loglik-gen}.

However, there are many other intuitive choices for \eqref{loglik-gen}, which also depend on the model. One might want to use the inverse of $\tilde{O}$ (if invertible), analogous to use of the inverse covariance matrix an standard multivariate analysis. 

Smoothing and penalization can be achieved by the choice of the $f_k$s and $Q$ and $\phi$. The choice of penalization in the context of smoothing is a delicate matter with good discussions in   \cite{berlinetAgnan} and \cite{ramsay}. 
In principle, any suitable embedding can be used for statistical inference. The default choice would be to use indicator functions; ie. $f_k = 1_{A_k}$, where $(A_k)_{k=1}^m$ is  an equipartition of $U$, such that $u_k \in A_k$. 
There are other interesting alternatives such as radial basis functions and piecewise-linear embeddings. 

For multivariate Gaussian random variables with variance $M$, the penalisation term of the usual log-likelihood expression is given by $\log\det M$. For operators, we do not have a similar concept. This is a challenge when defining and discussing inference for Gaussian processes. In Section \ref{l201} we introduce the Fredholm determinant, used in functional analysis and mathematical physics.


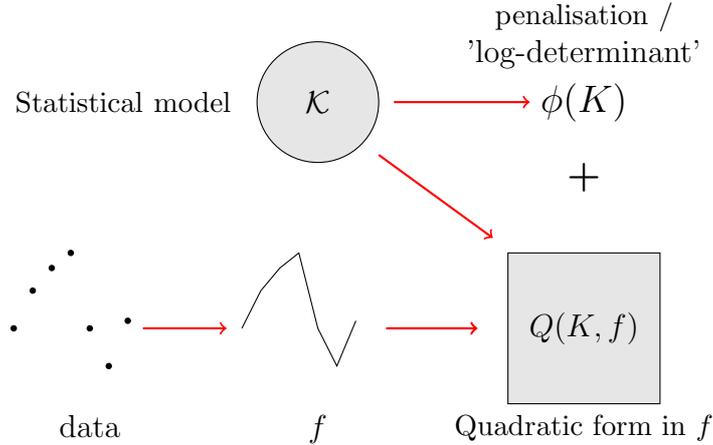
\begin{figure} \label{log-lik-fig}
	\[
	\begin{tikzpicture}
	\filldraw 
	(0,1) circle (1pt)
	(0.25,1.5) circle (1pt)
	(0.5,1.8) circle (1pt)
	(0.75,2) circle (1pt)
	(1,1) circle (1pt)
	(1.25,0.5) circle (1pt)
	(1.5,1.1) circle (1pt)
	(1, 0) node[below] {data};
	
	\draw (3,1) -- (3.25, 1.5) -- (3.5, 1.8) --
	(3.75, 2) -- (4, 1) -- (4.25, 0.5) -- (4.5, 1.1)
	(4, 0) node[below] {$f$};
	
	\draw [fill=gray!20, draw=black] (4, 4) circle[radius=0.8]
	(4, 4) node[circle, radius=0.8] (k) {}
	(4,4) node {${\mathcal{K}}$}
	(3,4) node[left, align=center] {\small Statistical model}
	
	(6.5, 0) rectangle(8.5, 2)
	
	(7.5, 1) node{ $Q(K, f)$ }
	(7.5, 0) node[below] {\small Quadratic form in $f$}
	(7.5, 4) node (k) {\large $\phi(K)$}
	(7.5, 4.3) node[above, align = center] {\small penalisation / \\ 'log-determinant'}
	(7.5, 3) node{\large \textbf{+}}
	;
	\draw[->, red,thick] (1.7 , 1) -- (2.8, 1);
	\draw[->, red,thick] (4.9 , 1) -- (6.1, 1);
	\draw[->, red,thick] (4.9 , 1) -- (6.1, 1);
	\draw[->, red,thick] (4.8, 3.3) -- (6.3, 2.2);
	\draw[->, red,thick] (5, 4) -- (k);
	\end{tikzpicture}
	\]
	\caption{Illustration of log-likelihood procedure, here with a linear interpolation and a single datum}
\end{figure}

\subsection{The functional nature of discrete data} 
Most time series data, including functional data, are fundamentally some kind of averages: for example, daily temperatures are averages, and measurement devices such as cameras require some time (the exposure time) to collect light. 
The science of measurement devices is not within scope of this article, but this illustrates that most time series data are in fact an underlying continuous process integrated wrt.\ some kernel: that is, our observed time series is actually applications of a \emph{functional}, which we here identify with a random variable. This is fundamentally a philosophical question, and one could argue that the arrow pointing from data to $f$ in Figure \ref{log-lik-fig} should actually be reversed.

\paragraph{Noise and signal} 
We remark that noise is an inherent feature of almost any kind of data, and that the signal-to-noise ratio decreases with increased resolution of observations.
This fits directly into the presented framework: the noise part corresponds to a diagonal (or multiplication) operator, whereas the signal part corresponds to an integral operator.

\section{Functional statistical modelling/inference on $L^2[0,1]$} \label{l201}
Here we focus on statistical inference for
the special case of $L^2[0,1]$, which is easily extended to $L^2[a,b]$ for general $a,b \in \R, a < b$.
In the following,  the term \emph{multivariate log-likelihood} (\mvl) shall refer to (twice) the negative log-likelihood for a zero-mean Gaussian random variable, ie.
$$
y \trans M^{-1} y + \log \det M
$$
where $y \sim N(0, M)$.

For the remainder of this section, let $O$ be a self-adjoint, invertible operator such that $O$ decomposes $O = D (\mathbf{I} + K) D$, where $D$ is a multiplication operator and $K$ is a self-adjoint integral operator with a continuous kernel. $D^2$ corresponds to the noise of the process and $DKD$ to the serial correlation. 

We may extend the mv. logL to functional data on $L^2[0,1]$ in a natural way using the following functional log-likelihood:
\begin{equation}
l = \brak{f, O^{-1} f} + \inte{0}{1} {\log D(t)}{t} + d(K) \label{01-pen-f}
\end{equation}
where $d(K)$ is the \emph{Fredholm determinant} of $K$. The Fredholm determinant was introduced by \cite{fredholm1903} as a criterion for solvability of integral equations and generalises the 'usual' determinant. The Fredholm determinant is defined by
$$
d(K) = \sum_{k=0}^\infty \frac{1}{n!} \int_a^b \dots \int_a^b \det [K(x_p, x_q) ]_{p,q = 1}^n \de x_1 \dots \de x_n
$$
and has the property that $I + K$ is invertible iff $d(K) \neq 0$. We see that the integral term in \eqref{01-pen-f} resembles the log-determinant of a diagonal matrix, and thus \eqref{01-pen-f} is the natural way to extend the multivariate log-likelihood to a functional log-likelihood. 

Expect for very basic cases, the Fredholm determinant is obviously very hard to calculate. However, a number of numerical methods exist; an extensive treatment can be found in \cite{bornemann}. In our work, we focus on the basic approximation by a matrix, because this corresponds to the covariance of a discretely observed Gaussian process, linking the operatorial framework to the standard inference for multivariate Gaussian random variables.

However, $d(K)$ as used in \eqref{01-pen-f} 'overpenalises' $K$ in comparison to multivariate 	log-likelihood, and this formulation is not asymptotically consistent with the matrix formulation even with the right normalisations, as shown in Example \ref{d+1} below.
This can be corrected by adjusting $d(K)$ by $1/n$: %
\begin{equation} \label{log-l-korrekt}
l = \brak{f, O^{-1} f} + \inte{0}{1} {\log D(t)}{t} + \tfrac{1}{n} d(K) 
\end{equation}
It should be noted that we use a different asymptotics wrt. the multivariate formulation than usual (see Theorem \ref{asymp-thm}).

\begin{eks} \label{d+1}
	The simplest example of a statistical model which mixes integral and multiplication operators is the class of operators 
	$$
	{K} = \alpha (\mathbf{I} + \delta \mathbf{1}), \quad \alpha > 0, \delta \geq 0
	$$
	on the space $L^2[0,1]$, where $\mathbf{1}$ is the integral operator with kernel $\mathbf{1}(s,t) \equiv 1$. This corresponds to a noise of variance $\alpha$ and a 'common' covariance of $\alpha \delta$, ie. a mixed model.	
	
	The inverse of $K = \alpha (\mathbf{I} + \delta \mathbf{1})$ is easily seen to be $\alpha^{-1} (\mathbf{I} - \frac{\delta}{\delta + 1} \mathbf{1})$. Following \eqref{01-pen-f}, the log-likelihood expression becomes
	\begin{equation}
	L(\alpha, \delta; f) = \brak{f, \alpha^{^-1} f - \frac{\delta}{\alpha(\delta + 1)} \mathbf{1} f} + \log(1 + \delta) + \log \alpha \label{L-1plusI}
	\end{equation}
	This can be optimized in an iterative procedure. For a fixed $\alpha$, the optimal $\delta$ is given by $\hat{\delta} = \alpha^{-1} \brak{f, \mathbf{1} f} -1$. 	
	
	When using an equidistant piecewise-constant embedding of data, the corresponding multivariate Gaussian model is the mixed model where
	$$
	\cov(X_i, X_j) = \begin{cases}
	\alpha(1 + \delta / n) & i = j \\ \alpha\delta/n & i  \neq j
	\end{cases}
	$$ 
	For this model one can show that the mv. logL is given by 
	$$
	l = \alpha^{-1} \sum_{i=1}^n x_i^2 - \alpha^{-1} \frac{\delta / n}{\alpha(\delta + 1)} (\sum_{i = 1}^n x_i)^2 + \log(1 + \delta) + n \log \alpha
	$$
	However, for a fixed $\alpha$, the optimal $\delta$ is given by $\hat{\delta} = \alpha^{-1}n^{-2} (\sum_{i = 1}^n x_i)^2  -1/n$ which converges to an unbiased estimator when $n \pil \infty$. 
	The term $(\sum_{i = 1}^n x_i)^2$ corresponds to $\brak{f, \mathbf{1} f}$, showing that the two estimators for $\delta$ differ by $\frac{n-1}{n}$.
\end{eks} 
The example above argues that the 'naive' estimator defined by \eqref{01-pen-f} is a bad choice; the bias on $\delta$ persists as $n \pil \infty$. 
The alternative \eqref{log-l-korrekt} works better, and asymptotically, the 'standard' multivariate likelihood and the functional likelihood \eqref{log-l-korrekt}  are equivalent as described in the following theorem:

\begin{thm} \label{asymp-thm}
	Let $O$ be a positive definite operator on $L^2[0,1]$ s.t. $O$ decomposes $O = D(I+K)D$. Here  $D$ is a multiplication operator such that $D(t)$ is continuous on $[0,1]$ with $0 < D_{\inf} < D(t) < D_{\sup} < \infty$	for all $t \in [0,1]$, and 
	%
	$K$ is assumed to be an integral operator with positive, symmetric and continuous kernel.
	
	Let $M_n$ be the finite matrix approximation of $O$, that is 
	$$
	(M_n)_{i,j} = \tfrac{1}{n} D(\tfrac{i- 1/2}{n}) K(\tfrac{i- 1/2}{n}, \tfrac{j- 1/2}{n}) D(\tfrac{j- 1/2}{n}) + D(\tfrac{i- 1/2}{n})^2 1_{i=j}
	$$
	Furthermore, let $y_n$ be a bounded sequence of vectors of length $n$, and define $f_n \in L^2[0,1]$ to be the piecewise-constant embedding of $y_n$ in $L^2[0,1]$:
	$$
	f_n(t) = (y_n)_i \for t \in (\tfrac{i-1}{n}, \tfrac{i}{n})
	$$	
	Then 
	$$
	\tfrac{1}{n}(y_n\trans M_n^{-1} y_n + \log\det M_n) - l(f_n, O)  \pil 0 \quad \for n \pil \infty
	$$
	In detail, we have the three following convergences:
	\begin{align*}
	\tfrac{1}{n}y_n\trans M_n^{-1} y_n - \brak{f_n, O^{^-1} f_n} &\pil 0 \\
	\log\det R_n &\pil d(K) \\
	\tfrac{1}{n}\log\det S_n &\pil \inte{0}{1}{\log D(t)}{t} ,
	\end{align*}
	where $S_n$ and $R_n$ are the following decomposition of $M_n$; $M_n = S_n R_n S_n$: 
	$$
	(R_n)_{i,j} = \tfrac{1}{n} K(\tfrac{i- 1/2}{n}, \tfrac{j- 1/2}{n}) + 1_{i=j}, \quad
	(S_n)_{i,j} = D(\tfrac{i- 1/2}{n}) 1_{i=j}
	$$
	Here $R_n$ corresponds to the integral operator to $K$, and $S$ corresponds to the multiplication operator $D$.
\end{thm}
\begin{proof}
	See appendix.
\end{proof}
Regarding the asymptotics, we should  think of the amount of information in data.  We also remark that $f_n \pil f$ for some $f \in L^2[0,1]$ is not considered important here, and so we do not impose any restrictions on $y_n$, besides being bounded. This is due to the remarks in Sections \ref{gw-basic} and \ref{gaus-ops} that instances of $O$ cannot be identified with a proper function(al) when $O$ is not compact.  However, we do actually have
$$
f_n \stackrel{L^2}{\pil} f\quad \Longrightarrow \quad l(f_n) \pil l(f) \quad  \for n \pil \infty 
$$
where $l$ is given by \eqref{L-1plusI}.

It should be remarked that Theorem \ref{asymp-thm} does not imply consistency for parameters of the statistical model.  
This would also be true even if the signal were fully observed, see e.g. \cite{sorensen_stoc} for examples.

\begin{eks} \label{bm-stoej} 
	Now consider \emph{Brownian motion with noise}, which can be viewed as an extension of Examples \ref{bm-eks} and \ref{d+1}.
	The statistical model is the class of operators on $L^2[0,1]$ given by
	$$
	K = \alpha^2 (\mathbf{I} + \lambda^2 \mathbf{B}), \quad \alpha, \lambda > 0
	$$
	where $\mathbf{B}$ is the integral operator with kernel $B(s,t) = \min(s,t)$.
	
	Some tedious calculations will show that the Fredholm determinant is given by:
	$$
	d(\lambda^2 \mathbf{B}) = \sum_{k = 0}^{\infty} \frac{1}{(2k)!} \lambda^{2k} = \cosh(\lambda)	
	$$
	so following \eqref{01-pen-f}, the log-likelihood becomes:
	$$
	\brak{f, K^{-1} f} + \log \alpha + \tfrac{1}{n} \log \cosh \lambda, \quad f \in L^2[0,1]
	$$
	where $n$ is the number of observations (which alternatively may be interpreted as a penalisation parameter).
	
	If $f \in L^2[0,1]$, one can show that $\brak{f, K^{-1} f}$ becomes $\alpha^{-2}\inte{0}{1}{g(t)^2}{t}$ where $g(t)$ is given by the forward integral equation:
	\[
	g(t) = f(t) - \lambda \inte{0}{t} {\tanh(\lambda s) g(s)} {s}
	\]
	Up to approximation error, $g(t)$ can be calculated in linear time of the number of grid points of the integral, which allows for fast calculations of the likelihood.
	
\end{eks}

We leave the (remaining) details of statistical modelling for future work, but note that two promising classes of operators for modelling are multiplication operators and triangular operators.
Together, these classes constitute an algebra, and many statistical models including those of Examples \ref{d+1} and \ref{bm-stoej} can be decomposed into such operators. 

\section{Discussion} \label{sect-disc}
This work presents a new direction for the theory of infinite-dimensional Gaussian random variables, which can be utilized in new and existing statistical models. 
Deviating from the traditional approach that a Gaussian process should be a function in a Hilbert space opens new possibilities, for instance that the covariance can be any operator on said Hilbert space.

From the modest background space of an \iid sequence of $N(0,1)$ random variables,   we have defined a comprehensive framework for Gaussian processes.
This encompasses not only standard continuous Gaussian processes and deterministic stochastic integration, but also Gaussian 'noise', all which fit well into the inferential framework presented in Section \ref{inference-sec}. 
As demonstrated in Proposition 5 and Remark 9, Gaussian processes are generally not proper functionals, not even for compact operators. However, we also saw that we still can define the 'classical' value of a Gaussian process in many cases, with a similar variance expression cf. Remark \ref{kont-remark}.

Stochastic models are more interesting when we can use them in a statistical setting as discussed in Sections \ref{inference-sec} and \ref{l201}. Although we in this work consider instances of Gaussian processes to be
improper functionals, we admittedly use functions for inference, where we here advocate for the embedding of discrete data into a piecewise constant function. 
More can be elaborated on statistical applications and inference, we leave this for future work. 

The asymptotical result in Theorem \ref{asymp-thm} is quite unusual in sense that we let the signal-to-noise ratio decrease. Had this ratio remained constant, we would have been able to fully separate the signal from the noise as the number of observations increased to infinity. Instead, we 'balance' the amount of information between signal and noise, a different perspective for the interpretation of observed data.

To the author's best knowledge, this is the first work to propose Fredholm determinants for 'log-likelihood' expressions as done in Eqs.\ \eqref{01-pen-f} and \eqref{log-l-korrekt}. However, as discussed,  the 'pure' form \eqref{01-pen-f} is biased and does not align with standard multivariate theory, and we therefore use the corrected form \eqref{log-l-korrekt}.
We leave a discussion of other correction methods for future work.

Finally, it should be remarked that Gaussianity is only crucial when establishing the independence of the choice of basis, and it will certainly be possible to generalize much of this work to non-Gaussian processes.

\section*{Acknowledgements}
I am grateful to Associate Professor Bo Markussen (University of Copenhagen) for valuable ideas and comments to this manuscript. 

\bibliographystyle{plain}

\appendix

\section{Proof of Theorem \ref{asymp-thm}}
We show each of the three claims of with a lemma.

\begin{lemma}  \label{fredholm-konv}
	We want to show $\log\det R_n \pil d(K)$ for $n \pil \infty$.

Let $R_n = \mathbf{I} + Q_n$. By the von Koch formula, 
\begin{equation}
\det R_n  = \det(I + Q_n) = \sum_{k = 0}^\infty \frac{1}{k!}  \sum_{i_1, \dots, i_k} \det [(Q_n)_{i_p, i_q}]_{p,q = 1}^k
\end{equation}
where ${i_1, \dots, i_k} \in \{1, \dots ,n\}$. The determinants inside the sum terms are non-zero only if all indices are different (due to colinearity). 	
	
Define $K_n$ to be the integral operator with a piecewise constant kernel given by
$$
K_n(s,t) = K(\tfrac{1}{2n} + \tfrac{i-1}{n} , \tfrac{1}{2n} + \tfrac{j-1}{n} )
\quad \for s \in (\tfrac{i-1}{n}, \tfrac{i}{n}], 
t \in (\tfrac{j-1}{n}, \tfrac{j}{n}]
$$
Let $d_{nk}$ denote the $k$'th term  of the Fredholm determinant of $K_n$. 
We have that
$$d_{nk} = \int_{[0,1]^n} \det [K_n(x_p, x_q) ]_{p,q = 1}^k \de x_1 \dots \de x_k.$$
As $K_n$ is piecewise constant, the integrand is non-zero only if $x_1, \dots , x_k$ belong to different intervals on the form $(\tfrac{i-1}{n}, \tfrac{i}{n}]$. From this it easily follows that
$$
d_{nk} = \frac{1}{k!} \sum_{i_1, \dots, i_k} \det [(Q_n)_{i_p, i_q}]_{p,q = 1}^k
$$	
proving that the Fredholm determinant of $K_n$ is given by $d(K_n) = \det{R_n}$.
Now, if we can show that the k'th term of $d_{nk}$ converge to the $k$'th term, $d_k$, of $d(K)$ for all $k$, then $d(K_n) \pil d(K)$ by the dominated convergence theorem.

Since $K$ is continuous on a compact set, there exists a sequence $(\epsilon_n)_{n=1}^\infty, \epsilon_n \pil 0$ such that $|K_n(s,t) - K(s,t) | < \epsilon_n$, and an $R$ s.t. $|K(s,t)| < R$. Then 
\begin{multline*}
| \det [K(x_p, x_q) ]_{p,q = 1}^k - \det [K_n(x_p, x_q) ]_{p,q = 1}^k | = \\
\left|\sum_{\sigma \in S(k)} \left( \prod_{p=1}^k K(x_p, x_{\sigma(p)}) - 
\prod_{p=1}^k K_n(x_p, x_{\sigma(p)}) \right) \right| < \\
\sum_{\sigma \in S(k)} (R + \epsilon_n)^k - R^k
\end{multline*}
which tends to zero as $n \pil 0$. From this it follows that $d_{nk} \pil d_k$, showing the result. 
\end{lemma}

\begin{lemma} \label{ld-diag} We want to show that $\tfrac{1}{n}\log\det S_n \pil \inte{0}{1}{\log D(t)}{t} $ for $n \pil \infty$. 
	
Observe that since $S_n$ is a diagonal matrix, $$
\log\det S_n = \sum_{i=1}^n \log S_{ii} = \sum_{i=1}^n \log D(\tfrac{i-1/2}{n}).
$$ 
As $D: [0,1] \pil \R$ is continous and bounded from below, the function $t \mapsto \log D(t)$ is uniformly continous for $t \in [0,1]$. Therefore, we can approximate the integral $\inte{0}{1}{\log D(t)}{t}$ by the midpoint Riemann sum, ie.
$$
\inte{0}{1}{\log D(t)}{t} = \lim_{n \pil \infty} \frac{1}{n}  \sum_{i=1}^n \log D(\tfrac{i-1/2}{n}).
$$
Now the result follows. 
 
\end{lemma}

\begin{lemma} We want to show that the quadratic forms which define the inverses converge, ie. $\tfrac{1}{n}y_n\trans M_n^{-1} y_n - \brak{f_n, O^{-1} f_n} \pil 0 $. 
	
	Let $O_n$ be the following approximation of $O$:
	$$
	O_n = D_n(I + K_n) D_n
	$$
	where $D_n$ and $K_n$ are defined as in Lemma \ref{ld-diag} and \ref{fredholm-konv}, respectively. 
	Since $D_n \pil D$ and $K_n \pil K$ in norm, we have that $||O - O_n || <\epsilon_n$ where $\epsilon_n \pil 0$ for $n \pil \infty$. 	

Now, observe by straightforward calculations that 
$$
Y_n\trans (M^n)^{-1} Y_n = n \brak{f_n, O_n^{-1} f_n}
$$
due to the construction of $O_n$ and $f_n$.	
For any $x \in L^2[0,1]$, we have the inequality:
$$
\begin{gathered}
(1 + \min K(s,t)) ||x|| \leq  ||(I + K) x|| \leq  (1 + \max K(s,t)) ||x|| \\
(1 + \min K(s,t)) ||x|| \leq  ||(I + K_n) x|| \leq  (1 + \max K(s,t)) ||x||
\end{gathered}
$$

Let $x \in L^2[0,1]$. There now exists $g, g_n \in L^2[0,1]$ s.t. $(K+I)g = (K_n + I) g_n = x$. Define $x' = (K+I) g_n$. Then, $||x - x'|| < \epsilon_n y_n < \frac{\epsilon_n}{(1 + \min K(s,t)) ||} ||x ||$, and we hvae that  
$$
||g_n - g|| \leq \frac{||x - x'||}{1 + \min K(s,t)} < \frac{\epsilon_n ||g_n||}{1 + \min K(s,t)} \leq
\frac{\epsilon_n ||x||}{(1 + \min K(s,t))^2} 
$$
which tends to zero as $n \pil \infty$. Therefore
$| \brak{f_n, (\mathbf{I} + O_n)^{-1} f_n} - \brak{f_n, (\mathbf{I} + O)^{-1} f_n} | < 
\frac{\epsilon_n}{(1 + \min K(s,t))^2}  ||f_n||^2$, As $f_n$ is bounded, this converges to zero for $n \pil \infty$, showing the result. 
\end{lemma}

\end{document}